\documentclass[notitlepage,a4paper]{article}
\usepackage{amsthm}
\usepackage{amsopn}
\usepackage{amsmath}
\usepackage{mathtools}
\usepackage{fancyvrb}

\usepackage{amssymb}
\usepackage{mathrsfs}
\usepackage{stix}

\usepackage{fullpage}

\usepackage[colorlinks=true,linkcolor=blue,psdextra]{hyperref}
\usepackage{enumitem}

\allowdisplaybreaks

\def\Z{\mathbb{Z}}

\DeclareMathOperator{\Spec}{Spec}

\DeclareMathOperator{\gr}{gr}
\DeclareMathOperator{\LT}{LT}
\DeclareMathOperator{\Vir}{Vir}

\def\vac{|0\rangle}
\def\cA{\mathscr{A}}
\def\cB{\mathscr{B}}
\def\cC{\mathscr{C}}
\def\cD{\mathscr{D}}
\def\cE{\mathscr{E}}

\def\cL{\mathscr{L}}

\def\cP{\mathscr{P}}

\def\fg{\mathfrak{g}}

\makeatletter
\def\thmhead@plain#1#2#3{%
  \thmname{#1}\thmnumber{\@ifnotempty{#1}{ }\@upn{#2}}%
  \thmnote{ {\the\thm@notefont#3}}}
\let\thmhead\thmhead@plain
\def\swappedhead#1#2#3{%
  \thmnumber{#2}%
  \thmname{\@ifnotempty{#2}{~}#1}%
  \thmnote{ {\the\thm@notefont#3}}}
\let\swappedhead@plain=\swappedhead
\makeatother

\makeatletter
\def\th@definition{
  \thm@notefont{}%
  \normalfont
}
\makeatother

\newtheorem{thm}{Theorem}
\swapnumbers
\newtheoremstyle{exps}{\topsep}{\topsep}{}{0pt}{\bfseries}{.}{0pt}{}

\newtheorem*{thm*}{Theorem}
\newtheorem*{prop*}{Proposition}
\newtheorem{ques*}{Question}
\newtheorem*{lem*}{Lemma}

\newtheorem*{cor*}{Corollary}

\newtheorem*{notn*}{Notation}
\newtheorem*{rem*}{Remark}
\newtheorem{cor}[subsection]{Corollary}

\newtheorem{prop}[subsection]{Proposition}

\newtheorem{lem}[subsection]{Lemma}

\theoremstyle{definition}
\newtheorem*{defn*}{Definition}
\newtheorem*{exer*}{Exercise}
\newtheorem*{ex*}{Example}

\newtheorem*{problem*}{Problem}
\newtheorem{rem}[subsection]{Remark}
\newtheorem{nolabel}[subsection]{}

\theoremstyle{exps}

\numberwithin{equation}{subsection}
\setenumerate[1]{label={\alph*)}} 

\date{}
\title{The singular support of the Ising model}
\author{George E. Andrews \thanks{Penn State University, Department of Mathematics. \newline \indent \indent \footnotesize{\texttt{\href{mailto:gea1@psu.edu}{gea1@psu.edu}}}} \and Jethro van Ekeren \thanks{Instituto de Matem\'{a}tica e Estat\'{i}stica (GMA), UFF, Niter\'{o}i RJ, Brazil. \newline \indent \indent \footnotesize{\texttt{\href{mailto:jethrovanekeren@gmail.com}{jethrovanekeren@gmail.com}}}}  \and  Reimundo Heluani \thanks{Instituto de Matem\'{a}tica Pura e Aplicada, Rio de Janeiro, RJ, Brazil \newline \indent \indent \footnotesize{\texttt{\href{mailto:heluani@potuz.net}{heluani@potuz.net}}}}}

\begin{document}
\maketitle
\begin{abstract}
We prove a new Fermionic quasiparticle sum expression for the character of the Ising model vertex algebra, related to the Jackson-Slater $q$-series identity of Rogers-Ramanujan type and to Nahm sums for the matrix $\left( \begin{smallmatrix} 8 & 3 \\ 3 & 2 \end{smallmatrix} \right)$. We find, as consequences, an explicit monomial basis for the Ising model, and a description of its singular support. We find that the ideal sheaf of the latter, defining it as a subscheme of the arc space of its associated scheme, is finitely generated as a differential ideal. We prove three new $q$-series identities of the Rogers-Ramanujan-Slater type associated with the three irreducible modules of the Virasoro Lie algebra of central charge $1/2$. We give a combinatorial interpretation to the identity associated with the vacuum module.
\end{abstract}

\section{Introduction}

\begin{nolabel}
In \cite{li05} Li introduced a canonical decreasing filtration $\{ F_p V \}$ on
an arbitrary vertex algebra $V$. The associated graded $\gr_F V$ of $V$ with
respect to this filtration carries the natural structure of a
$\Z_{\geq 0}$-graded Poisson vertex algebra and its spectrum is called the \emph{singular support} of $V$. It is known that $\gr_F V$ is generated as a differential algebra by $R_V = \gr_F^0 V$. By \emph{differential algebra} we mean a commutative algebra with a derivation as defined in \cite{ritt}.  $R_V$ is a Poisson algebra, and in fact is just the Zhu $C_2$-algebra of $V$. The spectrum $X_V = \Spec(R_V)$ is a Poisson scheme known as the \emph{associated scheme} of $V$.

The \emph{arc algebra} $(JR, \partial)$ of a commutative algebra $R$ is the $\Z_{\geq 0}$-graded differential algebra freely generated by $R$. If $R$ is a Poisson algebra, $JR$ acquires the natural structure of a Poisson vertex algebra. One thus has in general a surjection
\begin{equation}\label{eq:pi}
\pi: JR_V \twoheadrightarrow \gr_F V
\end{equation}
of Poisson vertex algebras \cite{arakawa}. We say that a vertex algebra $V$ is \emph{classically free} if $\pi$ is an isomorphism, reflecting the fact that $\gr_F V$, which is to be thought of as the classical limit of $V$, is freely generated as a Poisson vertex algebra.
\label{no:classically-free}
\end{nolabel}

\begin{nolabel}There are several classes of examples of classically free vertex algebras. 
Universal enveloping vertex algebras of (linear and non-linear) Lie conformal algebras are classically free. Thus the universal affine vertex algebra $V^k(\fg)$ associated with a finite dimensional semisimple Lie algebra $\fg$ is classically free, as is the universal affine $W$-algebra $W^k(\fg, f)$ associated with $\fg$ and a nilpotent element $f \in \fg$. In general rational vertex algebras are not expected to be classically free. We have for example
\begin{thm*}[{ {\cite{eh2018}}}]
Let $p$ and $p'$ be coprime positive integers satisfying $p' > p \geq 2$ and let $V=\Vir_{p,p'}$ be the simple vertex algebra associated with the $(p, p')$ Virasoro minimal model of central charge $c = 1 - 6 \tfrac{(p-p')^2}{pp'}$. Then $V$ is classically free if and only if $p=2$.
\end{thm*}
Other examples of rational, $C_2$-cofinite, classically free vertex algebras include the simple affine vertex algebra $V_k(\mathfrak{sl}_2)$ at positive integer level $k$ \cite{eh2018}. It is expected that $V_k(\fg)$ is classically free for all simple $\fg$ and $k \in \Z_{\geq 0}$. The simple vertex algebra $V_k(\mathfrak{sl}_2)$ at admissible level $k = -2 + p/p'$ ($p, p' \geq 2$ coprime) is known not to be classically free whenever $p' > p > 2$, and is expected to be classically free for the boundary admissible levels $k = -2 + 2/p'$. Arakawa and Linshaw have provided another example of a non classically free vertex algebra in \cite{arakawa-linshaw-arc}.
\label{no:non-classicallyfree-examples}
\end{nolabel}
\begin{nolabel}If $V$ is not classically free then the differential ideal $\ker
\pi \subset \gr_F{V}$ is nonzero, and the singular support of $V$ is a closed
proper subscheme of the arc space of $X_V$ with ideal sheaf given by $\ker \pi$.
It was asked in \cite{arakawa-linshaw-arc}, for the case of $V = \mathcal{W}_3$,
whether $\ker \pi$ is finitely generated as a differential ideal. The question
of finite generation of $\ker \pi$ also has an important application in bounding
dimensions of chiral homology groups of elliptic curves \cite{eh2018}. In this
note we study the structure of the singular support of the Virasoro minimal
model $\Vir_{3,4}$, also known as the Ising model.
\end{nolabel}
\begin{nolabel}
The singular support of a \emph{conformal} vertex algebra (i.e., of a VOA) acquires a $\Z_{\geq 0}$-grading induced by conformal weight, since the latter grading is compatible with the Li filtration. We write $\chi_V(q)$ for the graded dimension $\sum_{n\in \Z_{\geq 0}} \dim(V_n) q^n$ or \emph{character} of $V$, and similarly $\chi_{\gr_F V}(q)$ for the graded dimension of $\gr_F V$. One has on general grounds that $\chi_{\gr_F V}(q) = \chi_V(q)$. By taking account of the canonical grading of $\gr_F V$ as well, we can introduce a two-variable graded dimension
\begin{equation}\label{eq:char.gr.two}
\chi_{\gr_F V}(t, q) = \sum_{m, n \in \Z_{\geq 0}} \dim(\gr_F^m V_n) t^m q^n,
\end{equation}
which we may think of as a refinement of the character.

There are many well known formulas for the character of $\Vir_{3,4}$, for example:
\begin{equation}
\begin{aligned}
\chi_{\Vir_{3,4}}(q)
&= \prod_{n=1}^\infty \frac{1}{1-q^n} \sum_{m \in \Z} \left( q^{12m^2+m} - q^{12m^2+7m+1} \right) \\
&= \frac{1}{2} \left( \prod_{m=1}^\infty \left( 1+q^{m-1/2} \right) + \prod_{m=1}^\infty \left( 1-q^{m-1/2} \right) \right) \\ 
&= \sum_{m=0}^{\infty} \frac{q^{2 m^2}}{(1-q)(1-q^2)\dots(1-q^{2m})} \\ 
&= \prod_{k=1}^\infty \frac{(1+q^{8k - 5})(1+q^{8k-3})(1-q^{8k})}{1-q^{2k}}.
\end{aligned}
\label{eq:alternatives}
\end{equation}
The first of these expressions is the $(p, p') = (3, 4)$ case of the
Feigin-Fuchs character formula for the $(p, p')$ Virasoro minimal model
\cite{feigin-fuchs} (see Section \ref{no:virasoro} below). It derives from the
BGG-type resolution of modules over the Virasoro Lie algebra. The second expression
is directly implied by an isomorphism between $\Vir_{3,4}$ and the even
subalgebra $F_{\overline{0}}$ of the \emph{neutral free fermion} vertex
superalgebra $F$. The third expression for $\chi_{\Vir_{3,4}}(q)$ is obtained
from the second via a classical identity of Euler \cite[{Corollary
2.2}]{andrews-book}. The equality between the first and the fourth lines is a special case of Watson's \emph{quintuple product identity} \cite{bailey}. The equality between the third and fourth lines is known as the \emph{Jackson-Slater identity}: it appeared as identity (39) in Slater's famous list \cite{slater}, where it was proved by an application of the method of Bailey pairs, and it had appeared earlier in a disguised form in \cite{jackson}.

These $q$-series identities have combinatorial interpretations. For instance the fourth expression of \eqref{eq:alternatives} is the generating function $\sum_{n=0}^\infty a_n q^n$ for $a_n$ the number of partitions of $n$ into parts congruent to $\pm 2,\pm 3, \pm 4$ and $\pm 5$ modulo $16$. An interpretation of the third expression of \eqref{eq:alternatives} was given by Hirshhorn as the generating function for the number of partitions $[\lambda_1, \ldots, \lambda_m]$ of $n$ satisfying the \emph{difference conditions} \cite{hirschhorn}
\[
\lambda_m \geq 2, \quad \lambda_{m-1} - \lambda_m \geq 0, \quad \lambda_{m-2} - \lambda_{m-1} \geq 2, \quad \lambda_{m-3} - \lambda_{m-2} \geq 0, \dots.
\]
Further combinatorial interpretations have been given by Subbarao \cite{subbarao} and  Ribeiro \cite{ribeiro} (see also \cite{santos}).
The following striking formula was conjectured in \cite{kedem} and proved in \cite{warnaar}
\begin{equation}\label{eq:e8}
\chi_{\Vir_{3,4}}(q) = \sum_{k = (k_1, k_2, \ldots, k_8) \in \Z^8_{\geq 0}} \frac{q^{k^T C_{E_8}^{-1} k}}{(q)_{k_1} \dots (q)_{k_8}}.
\end{equation}
Here $C_{E_8}$ is the Cartan matrix of the simple Lie algebra $E_8$. In general
$(q)_n$ denotes the $q$-Pochhammer symbol $(q)_n = (1-q)\dots(1-q^n)$. Sums, like \eqref{eq:e8}, of the general form
\begin{equation} \label{eq:nahm}
\sum_{k = (k_1,\dots,k_n) \in \Z^n_{\geq 0}} \frac{q^{\frac{1}{2}k^T A k + k^T B + C}}{(q)_{k_1}\dots(q)_{k_n}},
\end{equation}
where $A$ is symmetric $n \times n$ positive definite matrix over $\mathbb{Q}$, $B \in \mathbb{Q}^n$ and $C \in \mathbb{Q}$, are sometimes referred to as \emph{Nahm sums} \cite{calegari}. Sums, like the third line of \eqref{eq:alternatives}, of the form \eqref{eq:nahm} but taken over $k \in \Z_{\geq 0}^n$ satisfying congruence conditions, as well as linear combinations of such sums, are sometimes referred to as (Fermionic) quasiparticle sums in the physics literature. Our first main result is a new Fermionic quasiparticle formula for $\chi_{\Vir_{3, 4}}(q)$.
\end{nolabel}
\begin{thm}
\begin{equation}\label{eq:q-ident}
\chi_{\Vir_{3, 4}}(q) = \sum_{(k_1,k_2) \in \Z_{\geq 0}^2} \frac{q^{4k_1^2+3k_1 k_2+k_2^2}}{(q)_{k_1} (q)_{k_2}} \left( 1-q^{k_1}+q^{k_1+k_2} \right). 
\end{equation}
\label{thm:q-ident}
\end{thm}
\begin{nolabel}
In this note we prove several structural results about the singular support of the Ising model $\Vir_{3, 4}$. In particular we show that $\Vir_{3, 4}$ is the first example of a non classically free vertex algebra
for which $\ker \pi$ is finitely generated as a differential ideal.

Let $V = \Vir_{3,4}$ be the Ising model. Its central charge is $c=\tfrac{1}{2}$. The Zhu $C_2$-algebra of $V$ is $R_V = \mathbb{C}[L]/(L^3)$. The arc algebra $JR_V$ is therefore the quotient of the polynomial algebra
\begin{equation*}
J\mathbb{C}[L_{-2}] = \mathbb{C}[L_{-2},L_{-3},\dots],
\end{equation*}
with its $\Z_{\geq 0}$-grading defined by $\deg(L_{-n}) = n$, by the differential ideal $(L_{-2}^3)_\partial$. Here the derivation $\partial$ is defined by $\partial(L_{-n}) = (n-1) L_{-n-1}$, and by differential ideal generated by a set $\{f_1, \ldots, f_n\}$ of elements we mean the ideal generated by all derivatives $\partial^k f_i$. Our second main result is the following 
\label{no:main-result}
\end{nolabel}
\begin{thm} Let $V = \Vir_{3,4}$ be the Ising model vertex algebra. We have an isomorphism of Poisson vertex algebras
\begin{equation}
JR_V \bigl/ ( b )_{\partial} \similarrightarrow \gr_F V \qquad \text{where} \quad b = \frac{1}{6}L_{-5} L_{-2}^2 +  L_{-4}L_{-3}L_{-2}.
\label{eq:isomvir}
\end{equation}
\label{thm:main}
\end{thm}
\begin{nolabel}
In fact we find a surjective morphism \eqref{eq:isomvir} for $V = \Vir_{3, p'}$ in general, where now $b$ is an explicit element of degree $2p'+1$. However for $p' \geq 5$ this morphism is not an isomorphism. For $p'=5$ for instance the morphism fails to be injective in degrees $\geq 19$. See Section \ref{sec:Preliminaries}.
\end{nolabel}
\begin{nolabel}
To explain the origin of \eqref{eq:q-ident} we return to the morphism of Theorem \ref{thm:main}, or more precisely to the corresponding bigraded surjection
\begin{equation}\label{eq:Rbig.quotient}
J\mathbb{C}[L_{-2}] \bigl/ ( a, b )_{\partial} \twoheadrightarrow \gr_F V, 
\end{equation}
where $a = L_{-2}^3$ and $b$ is as in \eqref{eq:isomvir}.
The graded algebra $J\mathbb{C}[L_{-2}]$ has a basis consisting of monomials
\begin{equation}
L_\lambda = L_{-\lambda_1} L_{-\lambda_2} \dots L_{-\lambda_m},
\label{eq:commut-monomials}
\end{equation}
parametrized by partitions $[\lambda_1, \ldots, \lambda_m]$ into parts $\lambda_1 \geq \ldots \geq \lambda_m \geq 2$. To compute the graded dimension of the quotient of $J\mathbb{C}[L_{-2}]$ by $I = (a, b)_\partial$, or equivalently the Hilbert series of $I$, it suffices to choose a monomial ordering in $J\mathbb{C}[L_{-2}]$ and compute the Hilbert series of the leading term ideal $\LT(I)$ of $I$. The task is then reduced to ennumerating partitions which satisfy certain generalized difference conditions. For instance the partition corresponding to any monomial multiple of $a$ contains $[2, 2, 2]$ and so this partition is to be excluded. Similarly the leading term (if we adopt the grevlex monomial order) of a multiple of $\partial^{3(p-2)}a$ contains $[p, p, p]$, etc. See Section \ref{sec:Preliminaries} for details.

Here and below we say that a partition $\lambda = [\lambda_1,\dots,\lambda_m]$ \emph{contains} another partition $\mu = [\mu_1,\dots,\mu_k]$ if the entries of $\mu$, counted with multiplicity, are contained among the entries of $\lambda$, i.e., if $\{\mu_1,\ldots, \mu_k\} \subseteq \{\lambda_1,\dots,\lambda_m\}$ in the sense of multisets. We say that $\lambda$ \emph{avoids} $\mu$ if $\lambda$ does not contain $\mu$.
\end{nolabel}
\begin{thm}
Let $\cP(n)$, be the set of partitions $\lambda = [\lambda_1,\dots,\lambda_m]$ of $n$ such that $\lambda_m \geq 2$ and $\lambda$ avoids the following partitions
\begin{equation}
\begin{aligned}
&
\left.
\begin{aligned}
&{[}p,p,p], && [p+1,p,p], &&[p+1,p+1,p] \\ 
&[p+2,p+1,p], && [p+2,p+2,p],  \\ 
\end{aligned} 
\right\} && p \geq 2 \\ 
&[p+2,p,p], &&  p \geq 3.  \\ 
&\left.
\begin{aligned}
&[p+3,p+3,p,p],&& [p+4,p+3,p,p], \\
&[p+4,p+3,p+1,p], && [p+4,p+4,p+1,p], 
\end{aligned}
\right\} && p \geq 2 \\ 
&[p+6,p+5,p+3,p+1,p], &&  p \geq 2 \\ 
&\begin{aligned}
&[5,4,2,2], && [7,6,4,2,2], & \\ 
&[7,7,4,2,2], && [9,8,6,4,2,2].
\end{aligned}
\end{aligned}
\label{eq:horrible-set}
\end{equation}
Let $V = \Vir_{3,4}$, then
\begin{enumerate}
\item \label{thm:part-a} for all $n \geq 0$ the set $\left\{ L_\lambda \right\}_{\lambda \in \cP(n)}$ is a basis of $V_n$,
\item \label{thm:part-b} for all $n \geq 0$, the number of partitions $\lambda \in \cP(n)$ equals
the number of partitions of $n$ with parts congruent to $\pm 2, \pm 3, \pm 4$
and $\pm 5$ modulo $16$,
\item \label{thm:part-c} if we denote by $p(n, m)$ the number of partitions in $\cP(n)$ into exactly $m$ parts, then the generating function of these cardinalities is given by the quasiparticle sum
\begin{equation}\label{eq:pmn.gen.function}
P(t,q) = \sum_{m, n \in \Z_{\geq 0}} p(n, m) t^m q^n = \sum_{(k_1,k_2) \in \Z_{\geq 0}^2} t^{2k_1+k_2} \frac{q^{4k_1^2+3k_1 k_2+k_2^2}}{(q)_{k_1} (q)_{k_2}} \left( 1-q^{k_1}+q^{k_1+k_2} \right).
\end{equation}
\end{enumerate}
\label{thm:partition-theorem}
\end{thm}
\begin{nolabel}
The strategy to prove Theorems \ref{thm:q-ident}--\ref{thm:partition-theorem} is as follows.
Part \ref{thm:part-c} of Theorem \ref{thm:partition-theorem} is proved in Section \ref{sec:partition}. 
Its proof proceeds essentially by finding a system of $q$-difference equations satisfied by the generating function $P(t, q)$ and showing that the right hand side of \eqref{eq:pmn.gen.function} satisfies the same system. The generating function for the cardinality of $\cP(n)$ is obtained by specializing $P(t, q)$ to $t=1$, and part \ref{thm:part-b} of Theorem \ref{thm:partition-theorem} follows at once from part \ref{thm:part-c} together with Theorem \ref{thm:q-ident} and the final line of equation \eqref{eq:alternatives}.

Theorem \ref{thm:q-ident} is proved in Section \ref{sec:polynomial-ident}. We
find two families of polynomials depending on a parameter $n \in \mathbb{Z}_+$
and show that they are equal by proving that they satisfy the same recurrence
equation. Their limits $n \rightarrow \infty$ give \eqref{eq:q-ident}. 

In Section \ref{sec:PBW} we show that, for each partition $\lambda$ containing one of the partitions \eqref{eq:horrible-set}, the monomial $L_\lambda$ is the leading monomial of some element of $I = (a,b)_\partial$. Since the morphism 
\eqref{eq:Rbig.quotient} is a surjection, to establish 
Theorem \ref{thm:main} it is enough to prove that the specialization of $P(t, q)$ to $t=1$ coincides with $\chi_{\gr_F \Vir_{3,4}}(q)$. This follows from Theorem \ref{thm:partition-theorem} \ref{thm:part-c} and Theorem \ref{thm:q-ident}, since the graded character of $\Vir_{3,4}$ and of its associated graded $\gr_F \Vir_{3,4}$ coincide. Finally, part \ref{thm:part-a} of Theorem \ref{thm:partition-theorem} follows
from part \ref{thm:part-c} and Theorem \ref{thm:main}. 
\label{no:strategy}
\end{nolabel}

As a corollary of Theorems \ref{thm:q-ident}--\ref{thm:partition-theorem} we deduce the following result on the structure of $\gr_F \Vir_{3, 4}$. In particular the bigraded dimension $\chi_{\gr_F \Vir_{3, 4}}(t, q)$ is obtained naturally from \eqref{eq:q-ident}. We remark that, by contrast, there seems to be no natural way to introduce $t$ into any of the formulas \eqref{eq:alternatives} and \eqref{eq:e8} for $\chi_{\Vir_{3, 4}}(q)$ in such a way as to obtain $\chi_{\gr_F \Vir_{3, 4}}(t, q)$.

\begin{cor}
Let $V = \Vir_{3, 4}$, the bigraded character of the singular support of $V$ is given by 
\begin{equation}\label{eq:q-t-ident}
\chi_{\gr_F V}(t, q) = P(t^{-2}, tq), 
\end{equation}
where $P(t, q)$ is the generating function \eqref{eq:pmn.gen.function}.
\label{cor:basis}
\end{cor}
\begin{nolabel}
Other monomial bases have been found for Virasoro minimal models. In \cite{ff1} monomial bases for $(2,p')$ Virasoro minimal models are constructed and used to give a new proof of the Andrews-Gordon identity. It was remarked there that describing the ideal in other cases should lead to interesting partition identities. In \cite{jimbo} the authors construct a monomial basis for all minimal models $\Vir_{p,p'}$ with $1 < p'/p < 2$. However the \emph{monomials} in their setting are rather different in that they refer to the action of intertwining operators and not the Virasoro modes themselves. In the case of the Ising model this corresponds to the classical Fermionic Fock space realization.
\label{no:jimbo}
\end{nolabel}

\begin{nolabel}
The normalized character $q^{-1/48} \chi_{\Vir_{3, 4}}(q)$ is a modular function. In \cite{nahm} Nahm conjectured (roughly speaking) that for any matrix $A$ for which the sum \eqref{eq:nahm} has appropriate asymptotic behaviour there exists some choice of $B$ and $C$ for which the sum is modular. (See Section \ref{no:nahmsum2} below.) The matrix $A = \left(\begin{smallmatrix}
8 & 3 \\ 3 & 2 \end{smallmatrix}\right)$ which figures in \eqref{eq:pmn.gen.function} above, has appeared in Terhoeven's list of matrices $A$ for which \eqref{eq:nahm} displays the correct asymptotics \cite{terhoeven}. But until now, despite computer searches \cite[\S 3.B.d)]{zagier}, no modular Nahm sum associated with the matrix had been found. Our example shows that modularity can be achieved by considering more general quasiparticle sums.
\label{no:nahmsum1}
\end{nolabel}

\begin{nolabel} The same techniques developed in this article can be used for other vertex algebras or even vertex algebra modules to yield new $q$-series and partition identities. For example, the Ising model $\Vir_{3,4}$ has three irreducible modules, $V_0 = V$, $V_{1/2}$ and $V_{1/16}$ where the subscript labels the degree of the unique singular vector. The (unnormalized) characters of $V_{1/2}$ and $V_{1/16}$ are given by 
\begin{equation}
\begin{aligned}
\chi_{V_{1/2}}(q) &=   
 \frac{1}{2} \left( \prod_{m=1}^\infty \left( 1+q^{m-1/2} \right) - \prod_{m=1}^\infty \left( 1-q^{m-1/2} \right) \right)   
=q^{1/2} \sum_{k \geq 1} \frac{q^{2 k^2 - 2k}}{(q)_{2k-1}}, \\ 
\chi_{V_{1/16}}(q) &= \prod_{m = 1}^\infty \left( 1+q^{m} \right) = \sum_{k \geq 0} \frac{q^{\frac{k(k+1)}{2}}}{(q)_{k}}.  
\end{aligned}
\end{equation}
In addition to Theorem \ref{thm:q-ident} pertaining to the character of $V_0$, we have
\begin{thm}  The (unnormalized) characters of $V_0$, $V_{1/2}$ and $V_{1/16}$ are given by 
\begin{equation}
\begin{aligned}
\chi_{V_{0}}(q) &= \sum_{k_{1}, k_{2} \geq 0} \frac{q^{4k_1^2 +3 k_1 k_2 + k_2^2 }}{(q)_{k_1}(q)_{k_2}} \left( 1 - q^{4k_1 + 2k_2 + 1} \right), \\ 
\chi_{V_{1/2}}(q) &= q^{1/2} \sum_{k_{1}, k_{2} \geq 0} \frac{q^{4k_1^2 +3 k_1 k_2 + k_2^2 + 2k_1}}{(q)_{k_1}(q)_{k_2}} \left( 1 - q^{8k_1 + 4k_2 + 6} \right), \\ 
\chi_{V_{1/16}}(q) &=\sum_{k_{1}, k_{2} \geq 0} \frac{q^{4k_1^2 +3 k_1 k_2 + k_2^2}}{(q)_{k_1}(q)_{k_2}} \left( q^{k_1+k_2} + q^{4k_1 + k_2 + 1} \right). 
\end{aligned}
\end{equation}
\label{thm:modules}
\end{thm}
The first equation is easily seen to be equivalent to Theorem \ref{thm:q-ident}.
The proof of this Theorem is found in Section \ref{sec:polynomial-ident}.
\end{nolabel}

\begin{nolabel}[Acknowledgments] RH would like to thank S. Kanade, A.Milas and P. Santos for
helpful discussions around equation \eqref{eq:q-ident}. The authors would like
to thank P. Paule and A. Riese for access to the
\texttt{Mathematica} package \texttt{qMultiSum}. This project was funded by CNPq grants 409582/2016-0 and 303806/2017-6. 
\end{nolabel}

\section{Preliminaries and notation}\label{sec:Preliminaries}

\begin{nolabel}\label{no:virasoro}
In working with $q$-series identities the following notation is useful. The $q$-Poch\-hammer symbol is $(q)_n = \prod_{j = 1}^n (1-q^j)$. We also write $(q)_\infty$ for $\prod_{j = 1}^\infty (1-q^j)$. The $q$-binomial coefficient is defined to be
\begin{equation}
\binom{m}{n}_q = \frac{(q)_m}{(q)_{n} (q)_{m-n}},
\label{eq:binom}
\end{equation}
for $0 \leq n \leq m$ and $0$ otherwise.
\end{nolabel}

\begin{nolabel} We denote by $\cL$ the Virasoro Lie algebra, namely the vector space with basis
$\left\{ L_n\right\}_{n \in \Z} \cup \{C\}$ and Lie brackets given by 
\begin{equation}
\label{eq:virbrackets} [L_m, L_n] = (m-n) L_{m+n} + \frac{m^3-m}{12}\delta_{m,-n} C, \qquad [C, \cL] = 0. \end{equation}
We also write $\cL_+$ for the subalgebra of $\cL$ spanned by $\left\{ L_n \right\}_{n \geq -1} \cup \{C\}$. The one dimensional representation $\mathbb{C}_c$ of $\cL_+$ of central charge $c$ is defined by $C \mapsto c$ and $L_n \mapsto 0$. The induced $\cL$-module
\[
\Vir^{c} = U(\cL) \otimes_{U(\cL_+)} \mathbb{C}_{c},
\]
carries the structure of a (conformal) vertex algebra \cite{frenkel-zhu}. By the PBW theorem the vectors
\begin{equation} 
 L_{-n_1} L_{-n_2}\dots L_{-n_m} \vac, \qquad n_1 \geq n_2 \geq \dots \geq n_m \geq 2. 
\label{eq:pbw-basis}
\end{equation}
constitute a basis of $\Vir^c$. We introduce a $\Z_{\geq 0}$-grading on $\Vir^c$ by assigning the monomial \eqref{eq:pbw-basis} degree $\sum_{i=1}^m n_i$. The monomials of degree $n$ are parametrized by partitions of $n$ into parts of size greater than or equal to $2$.

For the special values of the central charge $c = c_{p, p'}$ where $p, p' \geq 2$ are coprime integers and
\[
c = c_{p,p'} = 1 - \frac{6(p-p')^2}{pp'},
\]
the vertex algebra $\Vir^c$ is not simple, with maximal ideal (equivalently maximal $\cL$-submodule) generated by the homogeneous \emph{singular vector} $v_{p, p'}$ of degree $(p-1)(p'-1)$. The simple quotient, denoted $\Vir_{p, p'}$, is a rational vertex algebra known as the $(p, p')$ \emph{Virasoro minimal model}. The case $(p, p') = (3, 4)$, of central charge $c_{3, 4} = \tfrac{1}{2}$, is known as the \emph{Ising model}. The singular vector $v_{3, 4} \in \Vir^{1/2}$ is given explicitly by
\begin{equation} v_{3,4} = L_{-2}^3 \vac + \frac{93}{64}L_{-3}^2 - \frac{27}{16}L_{-6} \vac - \frac{33}{8}L_{-4}L_{-2}\vac.
\label{eq:singular}
\end{equation}
The graded dimension, or \emph{character}, of $\Vir_{p, p'}$ is \cite{feigin-fuchs}
\begin{equation}\label{eq.minmod.character}
\chi_{\Vir_{p,p'}}(q) = \frac{1}{(q)_\infty} \sum_{m \in \Z} \left( q^{\frac{\left(2p p'm + p- p'\right)^2 - \left(p -p'\right)^2}{4p p'}} - q^{\frac{\left(2pp' m + p+p'\right)^2 - \left(p- p'\right)^2}{4p p'}} \right).
\end{equation}
\end{nolabel}

\begin{nolabel}
Let $V$ be a vertex algebra. We now recall the definition of the Li filtration on $V$ \cite{li05}. It is the decreasing filtration $\{ F_p V \}$, defined by letting $F_p V$ be the linear span of vectors of the form
\[
a^1_{(-n_1-1)} a^2_{(-n_2-1)}\dots a^k_{(-n_k-1)}\vac, \quad a^i \in V, \quad n_i \geq 0, \quad \sum_{i=1}^k n_i \geq p.
\]
The associated graded $\gr_F V$ is known as the singular support of $V$ and is a $\Z_{\geq 0}$-graded Poisson vertex algebra. As remarked in \ref{no:classically-free} the degree $0$ component $\gr_F^0 V$ coincides with Zhu's Poisson algebra $R_V = V/V_{(-2)}V$.

If $V$ is conformal then a choice of homogeneous strong generators $\{u^i\}$ permits the introduction of an increasing filtration $\{G^p V\}$, also introduced by Li \cite{li05}, defined by letting $G^p V$ be the linear span of vectors of the form
\[
u^{i_1}_{(-n_1-1)} \ldots u^{i_N}_{(-n_N-1)}\vac, \quad n_i \geq 0, \quad \sum_{j=1}^N d_{i_j} \leq p.
\]
(Here $d_i$ denotes the degree, or conformal weight, of $u^i$.) 
Both filtrations are compatible with conformal weight and in fact it was proved by Arakawa that \cite{arakawa}
\begin{align}\label{eq:Arakawa.relation}
F_p V_n = G^{n-p} V_n.
\end{align}
For $\Vir^{c}$, which has $\omega = L_{-2}\vac$ as strong generating set, $G^p V$ is essentially the PBW filtration. Hence $F_p \Vir^{c}_n$ is the linear span of vectors of the form \eqref{eq:pbw-basis} satisfying $2m \leq n - p$. It is clear that $R_{\Vir^{c}} \simeq \mathbb{C}[\omega]$ and we have by the PBW theorem 
\[
\gr_F \Vir^{c} \simeq J R_{\Vir^{c}} \simeq \mathbb{C}[L_{-2},L_{-3},\dots],
\]
where $L_{-2} = \omega$ and the derivation $\partial$ is given by $\partial(L_{-n}) = (n-1) L_{-n-1}$.

As explained in the introduction the Li filtration on a vertex algebra $V$
entails a refinement of the character $\chi_V(q)$ to the two-variable character
$\chi_{\gr_F V}(t, q)$ defined in \eqref{eq:char.gr.two}. In this article we
work primarily with the PBW filtration on vertex algebras rather than the Li
filtration, but due to \eqref{eq:Arakawa.relation} it is easy to convert
generating functions from one to the other. As an application we have
\begin{proof}[Proof of Corollary \ref{cor:basis}]
Due to Theorem \ref{thm:partition-theorem} \ref{thm:part-a} and \ref{thm:part-c}
we know that the bigraded character $\chi_{\gr_F \Vir_{3,4}}(t,q)$ with respect
to the PBW filtration is given by $P(t,q)$ as in \eqref{eq:pmn.gen.function}.
Then \eqref{eq:Arakawa.relation} translates this into \eqref{eq:q-t-ident}.
\end{proof}

The filtrations on $\Vir_{p,p'}$ coincide with those induced by the quotient map from $\Vir^{c}$, where $c = c_{p,p'}$. In particular the quotient induces a surjection $R_{\Vir^{c}} \twoheadrightarrow R_{\Vir_{p,p'}}$. Let $s=(p-1)(p'-1)/2$. It is known that the coefficient of $L_{-2}^s\vac$ in the singular vector $v_{p,p'}$ is nontrivial \cite{feigin-fuchs}. All other monomials of the same degree lie in $F_1 \Vir_{p,p'}$ and hence $L_{-2}^s$ vanishes in $R_{\Vir_{p,p'}}$. Indeed, abusing notation, we have $R_{\Vir_{p,p'}} \simeq \mathbb{C}[L_{-2}]/(L_{-2}^s)$ and consequently
\[
JR_{\Vir_{p,p'}} \simeq \mathbb{C}[L_{-2},L_{-3},\dots]/(L_{-2}^s)_{\partial}.
\]
As in \eqref{eq:pi} there is a canonical surjection
\[
\pi : \mathbb{C}[L_{-2},L_{-3},\dots]/(L_{-2}^s)_{\partial} \twoheadrightarrow \gr_F \Vir_{p,p'}
\]
The graded dimension of the arc algebra has been computed in \cite{mourtada}. There it is shown that the ideal $(L_{-2}^s)_{\partial}$ has a Gr\"{o}bner basis whose leading terms are, translated to the present context, the monomials
\[
L_{-\lambda_1} \cdots L_{-\lambda_m} \in \mathbb{C}[L_{-2},L_{-3},\dots],
\]
for which the partition $\lambda = [\lambda_1, \ldots, \lambda_m]$ satisfies $\lambda_m \geq 2$ and the difference condition
\[
\lambda_i - \lambda_{i+s-1} \geq 2, \qquad \text{for $1 \leq i \leq m+1-s$}.
\]
Such partitions are counted by the left hand side of the following
Andrews-Gordon identity \cite{andrews-general}
\begin{equation}\label{eq:AG.identity}
\sum_{k = (k_1, \ldots, k_{s-1}) \in \Z_{\geq 0}^{s-1}} \frac{q^{\frac{1}{2} {k}^T G^{(s)} {k} + k^T B^{(s)}}}{(q)_{k_1} \cdots (q)_{k_{s-1}}} = \prod_{\substack{n \geq 1, \\ n\not\equiv 0, \pm 1 \bmod{2s+1}}} \frac{1}{1-q^n}.
\end{equation}
Here $G^{(s)}$ is the matrix with entries $G^{(s)}_{i, j} = 2\min\{i, j\}$ for $1 \leq i, j \leq s-1$ and $B^{(s)} = (1, 2, \ldots, s-1)$. For $p=2$ the Jacobi triple product identity reduces \eqref{eq.minmod.character} to the right hand side of \eqref{eq:AG.identity}. The $(2, 2s+1)$ Virasoro minimal models are thus classically free.
\label{no:arcising}
\end{nolabel}

\begin{nolabel} For the case of the Ising model $(p, p') = (3, 4)$ we have the following result.
\begin{lem*}
The kernel of the surjection
\[
\pi : \mathbb{C}[L_{-2},L_{-3},\dots]/(L_{-2}^3)_{\partial} \twoheadrightarrow \gr_F \Vir_{3,4}
\]
is a nontrivial graded differential ideal. Its lowest graded piece is the linear span of the degree $9$ vector
\[
b = \frac{1}{6} L_{-5}L_{-2}^2 + L_{-4}L_{-3}L_{-2}.
\]
\end{lem*}
\begin{proof}
It is easy to check that the lowest graded piece of $\ker\pi$ has degree $9$ and is $1$-dimensional by comparing the graded dimensions of $\Vir_{3, 4}$ and the arc algebra. Now let 
\[
w_{3,4} = L_{-5}L_{-2}L_{-2} \vac + 6 L_{-4}L_{-3}L_{-2} \vac \in F_{3} \Vir^{1/2}.
\]
Using \eqref{eq:virbrackets} and \eqref{eq:singular} we prove by direct computation
\begin{multline*}
w_{3,4} + \frac{256}{429} L_{-3}v_{3,4} - \frac{64}{429} L_{-1} L_{-2}v_{3,4} - \frac{31}{286} L_{-1}^{3}v_{3,4} = \\
\frac{27}{8} L_{-6}L_{-3}\vac +\frac{87}{4}L_{-7}L_{-2}\vac + \frac{147}{32}L_{-9}\vac - \frac{45}{16}L_{-5}L_{-4}\vac \in F_{5}\Vir^{1/2}. 
\end{multline*}
The Lemma follows applying $\pi$ to both sides of this equation and noting that the image of the LHS equals $\pi\bigl(w_{3,4}\bigr)$. 
\end{proof}
In fact by the same technique we may prove in general:
\begin{lem*}
Let $V = \Vir_{3, p'}$ where $p' \geq 4$. The kernel of the surjection
\[
\pi : \mathbb{C}[L_{-2},L_{-3},\dots]/(L_{-2}^{p'-1})_{\partial} \twoheadrightarrow \gr_F \Vir_{3,p'}
\]
is a nontrivial graded differential ideal. Its lowest graded piece is the linear span of the degree $2p'+1$ vector
\[
b^{(p')} = \frac{(9-2p')}{3(p'-2)} L_{-5}L_{-2}^{p'-2} + L_{-4}L_{-3}L_{-2}^{p'-3}.
\]
\end{lem*}

\begin{cor} Let $V = \Vir_{3,p'}$. There is a surjective morphism
\begin{equation*}
JR_V \bigl/ I \similarrightarrow \gr_F V,
\end{equation*}
where $I$ denotes the differential ideal $(b^{(p')})_\partial  \subset JR_V$.
\label{cor:main}
\end{cor}
As remarked in the introduction, the surjection of Corollary \ref{cor:main} is
not an isomorphism for $p' \geq 5$. For the case of the Ising model $V =
\Vir_{3, 4}$ we study the structure of the ideal $I$ in the following Sections.
\label{no:vanishesw34}
\end{nolabel}

\begin{nolabel}[Asymptotics of Nahm sums] \label{no:nahmsum2}
Let $A$ be an $n \times n$ positive definite symmetric matrix with rational coefficients, ${B} = (b_1, \ldots, b_n) \in \mathbb{Q}^n$ and $C \in \mathbb{Q}$. The formal power series
\begin{align}\label{eq:fermionic.sum}
f_{A, B, C}(q) = \sum_{{k} \in \Z_{\geq 0}^r} \frac{q^{\frac{1}{2}{k}^T A {k} + {k}^T {B} + C}}{(q)_{k_1} \cdots (q)_{k_r}}, 
\end{align}
converges, upon setting $q = e^{2\pi i \tau}$, to a holomorphic function $F(\tau)$ of $\tau \in \mathbb{H}$ the upper half complex plane. The asymptotic behaviour of $F(\tau)$ as $\tau = it \rightarrow 0$ along the positive imaginary axis is known to be \cite{vlasenko-zwegers}
\begin{equation}\label{eq:VZ-asymp}
F(i t) \sim e^{\frac{\alpha}{t}}, \quad \text{where} \quad \alpha = \sum_i \left( \frac{\pi^2}{6} - L(Q_i) \right).
\end{equation}
Here
\[
L(z) = \sum_{n=1}^\infty \frac{z^n}{n^2} + \frac{1}{2}\log(z)\log(1-z),
\]
is the Rogers dilogarithm function, and $(Q_1, \ldots, Q_n)$ is the unique solution of the system of equations
\begin{equation}\label{eq:Nahm.system}
1 - Q_i = \prod_j Q_j^{A_{ij}}, \qquad 0 < Q_i < 1 \: \text{ for } \: i = 1,\ldots, n.
\end{equation}

For example if $A = \begin{psmallmatrix}8 & 3 \\ 3 & 2\end{psmallmatrix}$ then
\[
Q_1 = \frac{1}{2}\left(\sqrt{2\sqrt{2}-1} + \sqrt{2} - 1\right)
\quad \text{and} \quad
Q_2 = \frac{2}{\sqrt{2\sqrt{2}-1} - \sqrt{2} + 3}.
\]
Then $\alpha = \frac{\pi^2}{12}$ can be deduced without difficulty from the functional equations satisfied by $L(z)$.


The normalized character $q^{-c_V/24} \chi_{V}(q)$ of a rational $C_2$-cofinite conformal vertex algebra is a modular function on $\Gamma_0(N)$ for some $N$ \cite{zhu}, \cite{dong-lin-ng}. Therefore the asymptotic behaviour of $\chi_V$ is always of the form $e^{\frac{\alpha}{t}}$ where $\alpha = \pi^2 g_V / 6$ for some $g_V \in \mathbb{Q}$ known as the \emph{effective central charge} of $V$. For the Ising model we have $g_V = c_V = 1/2$ and the normalized character is expressible as a linear combination of the Weber modular functions.


The condition $\alpha \in \pi^2 \mathbb{Q}$, where $\alpha$ is given by
\eqref{eq:VZ-asymp}, places a strong restriction on possible fermionic sum
representations of characters of vertex algebras and their modules. The matrix
$A = G^{(s)}$ appearing in connection with the Andrews-Gordon identity above
satisfies the condition and is associated with the vertex algebra $\Vir_{2,
2s+1}$. It was observed by Nahm that the rationality condition is closely
related to torsion elements in Bloch groups of number fields \cite{nahm}.
Terhoeven produced a list of matrices for which the rationality condition is
satisfied \cite{terhoeven}. For most of these matrices $A$ there exists a choice
of $B$ and $C$ for which \eqref{eq:fermionic.sum} becomes modular, but for other
examples such as $\begin{psmallmatrix}8 & 3 \\ 3 & 2\end{psmallmatrix}$, no such
choices had been found \cite{zagier}. Our Theorem \ref{thm:q-ident} shows that
modular candidates can be obtained allowing linear combinations of Nahm 
sums.
\end{nolabel}

\section{A partition identity} \label{sec:partition}
In this section we prove Theorem \ref{thm:partition-theorem} \ref{thm:part-c}.
We recall the set of partitions $\cP(n)$ introduced in Theorem \ref{thm:partition-theorem}. We denote by $p(n, m)$ the number of partitions $\lambda = [\lambda_1, \ldots, \lambda_m] \in \cP(n)$ into exactly $m$ parts and we write the generating function
\[
P(t, q) = \sum_{m, n} p(n, m) t^m q^n.
\]
In this section we prove fermionic sum expressions for $P(t, q)$. To achieve this we divide $\cP(n)$ into five disjoint subsets, we find recurrence relations among the cardinalities of these sets, these lead to a system of functional equations which we solve in Proposition \ref{prop:closed-formulas} to obtain $P(t, q)$. Specialization to $t=1$ yields the generating function \eqref{eq:q-ident} for the number of partitions in $\cP(n)$.
\begin{nolabel} It is not difficult to verify that $\cP(n)$ decomposes as the disjoint union
\[
\cP(n) = \cA(n) \amalg \cB(n) \amalg \cC(n) \amalg \cD(n) \amalg \cE(n),
\]
where, for $n \geq 0$, we define $\cA(0) = \{[]\}$, $\cB(2) = \{[2]\}$, $\cD(4) = \{[2, 2]\}$ and otherwise
\[ 
\begin{aligned}
\cA(n) &= \left\{ \lambda = [\lambda_1,\dots,\lambda_m] \in \cP(n) \: \big| \: 2<\lambda_m \right\}, \\ 
\cB(n) &= \left\{ \lambda = [\lambda_1,\dots,\lambda_m] \in \cP(n) \: \big| \: 2 = \lambda_m < \lambda_{m-1} - 1 \right\}, \\ 
\cC(n) &= \left\{ \lambda = [\lambda_1,\dots,\lambda_m] \in \cP(n) \: \big| \:2 = \lambda_m = \lambda_{m-1}-1 \right\}, \\
\cD(n) &= \left\{  \lambda = [\lambda_1,\dots,\lambda_m] \in \cP(n) \: \big| \: 2 = \lambda_m = \lambda_{m-1}  < \lambda_{m-2} - 2 \right\}, \\ 
\cE(n) &= \left\{  \lambda = [\lambda_1,\dots,\lambda_m] \in \cP(n) \: \big| \: 2= \lambda_m = \lambda_{m-1} = \lambda_{m-2} -2 \right\}.
\end{aligned}
\]
\label{no:subdivision}
\end{nolabel}
\begin{lem} Let $a(n,m)$ (resp. $b(n,m)$, \dots, $e(n,m)$) denote the number of partitions $\lambda \in \cA(n)$ (resp. $\cB(n)$, \dots, $\cE(n)$) into exactly $m$ parts. The following recursive formulas hold
\begin{equation}
\begin{aligned}
a(n,m) &= a(n-m,m)+b(n-m,m)+c(n-m,m)+d(n-m,m), \\ 
b(n,m) &= a(n-m-1,m-1) - d(n-2m,m-1), \\ 
c(n,m) &= b(n-2m+1,m-1) + d(n-2m,m-1), \\ 
d(n,m) &= b(n-m,m-1) - e(n-2m +1,m-1), \\ 
e(n,m) &= c(n-m,m-1).
\end{aligned}
\label{eq:recursion}
\end{equation}
\label{lem:recursion}
\end{lem}
\begin{proof}
Let $\lambda = [\lambda_1, \dots, \lambda_m] \in \cA(n)$ and consider $\mu = [\lambda_1
-1,\dots,\lambda_m -1] \in \cP(n-m)$. Notice that $\mu$ avoids $[4,2,2]$ since $[5,3,3]$ is one of the partitions excluded in the definition of $\cP$ (see the third line of \eqref{eq:horrible-set}). It follows that $\mu \in \cA(n-m) \amalg \cB(n-m) \amalg \cC(n-m) \amalg \cD(n-m)$. Conversely, if $\mu$ lies in the latter set then $\lambda = [\mu_1+1, \dots,\mu_m+1]$ lies in $\cP(n)$. The only conditions from \eqref{eq:horrible-set} that are not immediate are that $\lambda$ avoids $[5, 3, 3]$ and the four exceptional partitions listed there. But these partitions are indeed avoided by $\lambda$ because $\mu \notin \cE(n-m)$ and hence avoids $[4,2,2]$, and because $\lambda_m > 2$. Indeed this demonstrates that $\lambda \in \cA(n)$, and the bijection we have thus established proves the first equation in \eqref{eq:recursion}. The other equations are proved in a similar way.
\end{proof}
\begin{lem} Consider the formal power series
\[ A(t,q) = \sum_{n,m \geq 0} a(n,m) q^n t^m, \]
and define similarly $B(t,q), C(t,q), E(t,q), D(t,q)$. Then they satisfy the following functional equations:
\begin{equation}
\begin{aligned}
A(t,q) &= A(t q, q)+B(tq,q)+C(tq,q)+D(tq,q), & A(0,q) &= 1,\\ 
B(t,q) &= t q^2 A(t q,q) - t q^2 D(t q^2,q) , & B(0,q) &= 0, \\
C(t,q) &= t q B(t q^2,q) + t q^2 D(t q^2,q),& C(0,q) &= 0, \\ 
D(t,q) &= t q B(t q, q) - t q E(t q^2,q), & D(0,q) &=0,\\
E(t,q) &= t q C(t q, q) ,&  E(0,q) &=0. 
\end{aligned}
\label{eq:functional}
\end{equation}
\label{lem:functional}
\end{lem}
\begin{proof}
The functional equations are obtained by direct translation of the recurrence relations \eqref{eq:recursion} in terms of the power series we have introduced. 
\end{proof}
\begin{prop}
The unique solution of the system of functional equations \eqref{eq:closed-formulas} is
\begin{equation}
\begin{aligned}A(t,q) 
&=  \sum_{m \geq 0} \frac{t^m q^{m(m+1)}}{(q)_m} \sum_{k =0}^m t^k
q^{(k+1)m+2k^2} \binom{m}{k}_q, \\ 
B(t,q) &= \sum_{m \geq 1} \frac{t^m q^{m(m+1)}}{(q)_{m-1}} \sum_{k=0}^{m-1} t^k
q^{k(m+1)+2k^2} \binom{m-1}{k}_q, \\ 
C(t,q) &= \sum_{m \geq 2} \frac{t^m q^{m^2+1}}{(q)_{m-2}} \sum_{k=0}^{m-2} t^k q^{k(m+3)+2k^2} \binom{m-2}{k}_q, \\
D(t,q) &= \sum_{m \geq 2} \frac{t^m q^{m^2}}{(q)_{m-2}}\sum_{k\geq 0}^{m-2} t^k q^{k(m+2) + 2 k^2} \binom{m-2}{k}_q, \\
E(t,q) &= \sum_{m \geq 3}\frac{t^m q^{m^2-m+2}}{(q)_{m-3}}\sum_{k=0}^{m-3} t^k q^{k(m+3)+2k^2} \binom{m-3}{k}_q.
\end{aligned}
\label{eq:closed-formulas}
\end{equation}
\label{prop:closed-formulas}
\end{prop}
\begin{proof}
The uniqueness is automatic, it suffices to check that the expressions given in \eqref{eq:closed-formulas} satisfy the functional equations and initial conditions \eqref{eq:functional}. Directly from the definition we have:
\begin{equation}
A(tq,q) = \sum_{m \geq 0} \frac{t^m q^{m(m+1)}}{(q)_m} \sum_{k=0}^m t^k q^{(k+1)m+2k^2} \binom{m}{k}_q q^{m+k}. 
\label{eq:atq}
\end{equation}
Similarly:
\begin{equation}
\begin{aligned}
B(tq,q) &= \sum_{m \geq 1} \frac{t^m q^{m(m+1)}}{(q)_{m-1}} \sum_{k =0}^{m-1} t^k q^{(k+1)m + 2k^2} \binom{m-1}{k}_q q^{2k}  \\ 
&= \sum_{m \geq 0} \frac{t^m q^{m(m+1)}}{(q)_m} \sum_{k = 0}^{m} t^k q^{(k+1)m + 2k^2} \binom{m}{k}_q q^{2k} (1-q^{m-k}).
\label{eq:btq}
\end{aligned}
\end{equation}
For $C(t,q)$ and $D(tq,q)$ we perform a simple algebraic manipulation:
\begin{align}
\begin{split}
C(tq,q) &= \sum_{m \geq 2} \frac{t^m q^{m(m+1))}}{(q)_{m-2}} \sum_{k=0}^{m-2} t^{k} q^{k(m+3)+2k^2} \binom{m-2}{k}_q  q^{k+1} \\ 
&=\sum_{m \geq 1} \frac{t^m q^{(m+1)(m+2)}}{(q)_{m-1}} \sum_{k=0}^{m-1} t^{k+1} q^{k(m+4)+2k^2} \binom{m-1}{k}_q q^{k+1}\\ 
&=\sum_{m \geq 1} \frac{t^m q^{(m+1)(m+2)}}{(q)_{m-1}} \sum_{k=1}^{m} t^{k} q^{(k-1)(m+4)+2(k-1)^2} \binom{m-1}{k-1}_q q^{k}\\ 
&= \sum_{m \geq 1} \frac{t^m q^{m(m+1)}}{(q)_{m-1}} \sum_{k=1}^{m} t^k q^{(k+1)m+2k^2} \binom{m-1}{k-1}_q q^k,
\label{eq:ctq}
\end{split} \\ 
\begin{split}
D(tq,q) &= \sum_{m \geq 2} \frac{t^m q^{m(m+1))}}{(q)_{m-2}} \sum_{k=0}^{m-2} t^{k} q^{k(m+3)+2k^2} \binom{m-2}{k}_q  \\ 
&=\sum_{m \geq 1} \frac{t^m q^{(m+1)(m+2)}}{(q)_{m-1}} \sum_{k=0}^{m-1} t^{k+1} q^{k(m+4)+2k^2} \binom{m-1}{k}_q \\ 
&=\sum_{m \geq 1} \frac{t^m q^{(m+1)(m+2)}}{(q)_{m-1}} \sum_{k=1}^{m} t^{k} q^{(k-1)(m+4)+2(k-1)^2} \binom{m-1}{k-1}_q \\ 
&= \sum_{m \geq 1} \frac{t^m q^{m(m+1)}}{(q)_{m-1}} \sum_{k=1}^{m} t^k q^{(k+1)m+2k^2} \binom{m-1}{k-1}_q.
\label{eq:dtq}
\end{split}
\end{align}
Summing \eqref{eq:ctq} and \eqref{eq:dtq} and using 
\[\frac{1}{(q)_{m-1}} \binom{m-1}{k-1}_q = \frac{1}{(q)_m} \binom{m}{k}_q (1 - q^k), \]
we obtain 
\begin{equation}
C(tq,q)+D(tq,q) =  \sum_{m \geq 0} \frac{t^m q^{m(m+1)}}{(q)_m} \sum_{k=0}^{m} t^k q^{(k+1)m+2k^2} \binom{m}{k}_q (1-q^{2k}). 
\label{eq:cdtq}
\end{equation}
Adding \eqref{eq:atq}, \eqref{eq:btq} and \eqref{eq:cdtq} we obtain the first equation in \eqref{eq:functional}. 
\begin{equation*}
\begin{split}
\MoveEqLeft
t q^2 A(tq,q) - tq^2 D(tq^2,q) = \sum_{m \geq 0} \frac{t^{m+1}q^{m^2+3m+2}}{(q)_m} \sum_{k=0}^{m} t^k q^{k(m+1)+2k^2} \binom{m}{k}_q \\ 
&= - \sum_{m\geq 2} \frac{t^{m} q^{m^2 + 2m +2}}{(q)_{m-2}} \sum_{k=0}^{m-2} t^{k+1} q^{k(m+4) + 2k^2}\binom{m-2}{k}_q  \\ 
&= \sum_{m \geq 1} \frac{t^{m}q^{m^2+m}}{(q)_{m-1}} \sum_{k=0}^{m-1}t^k q^{km + 2k^2} \binom{m-1}{k}_q \\ 
&\quad - \sum_{m \geq 2} \frac{t^m q^{m^2+2m+2}}{(q)_{m-2}} \sum_{k=1}^{m-1} t^k
q^{(k-1)(m+4)+2(k-1)^2} \binom{m-2}{k-1}_q \\
&= \sum_{m\geq 1} \frac{t^m q^{m(m+1)}}{(q)_{m-1}} \sum_{k=0}^{m-1}t^k q^{k(m+1)+2k^2}\binom{m-1}{k}_q q^{-k} \\ 
&\quad - \sum_{m \geq 1} \frac{t^{m} q^{m(m+1)}}{(q)_{m-1}} \sum_{k=0}^{m-1} t^k q^{k(m+1)+2k^2} \binom{m-1}{k-1}_q q^{-k}(1-q^{m-k})  \\
&=\sum_{m\geq 1} \frac{t^m q^{m(m+1)}}{(q)_{m-1}} \sum_{k=0}^{m-1}t^k q^{k(m+1)+2k^2}\binom{m-1}{k}_q q^{-k} \\ 
& \quad - \sum_{m \geq 1} \frac{t^{m} q^{m(m+1)}}{(q)_{m-1}} \sum_{k=0}^{m-1} t^k q^{k(m+1)+2k^2} \binom{m-1}{k}_q (q^{-k}-1) = B(t,q) 
\end{split}
\end{equation*}
proving the second equation in \eqref{eq:functional}. 
\begin{equation*}
\begin{split}
\MoveEqLeft
tq B(t q^2,q)+t q^2 D(t q^2,q) = \sum_{m \geq 1} \frac{t^{m+1}q^{m^2+3m+1}}{(q)_{m-1}} \sum_{k=0}^{m-1} t^k q^{k(m+3)+2k^2}\binom{m-1}{k}_q \\
&\quad + \sum_{m\geq 2}\frac{t^{m+1} q^{m^2+2m +2}}{(q)_{m-2}} \sum_{k=0}^{m-2} t^k q^{k(m+4) + 2k^2} \binom{m-2}{k}_q  \\ 
&= \sum_{m \geq 2} \frac{t^m q^{m^2+m-1}}{(q)_{m-2}} \sum_{k=0}^{m-2} t^k q^{k(m+2)+2k^2} \binom{m-2}{k}_q \\ 
&\quad + \sum_{m \geq 3} \frac{t^{m} q^{m^2 +1}}{(q)_{m-3}} \sum_{k=0}^{m-3} t^{k} q^{k(m+3)+2k^2}\binom{m-3}{k}_q  \\ 
&=\sum_{m \geq 2} \frac{t^m q^{m^2+1}}{(q)_{m-2}} \sum_{k=0}^{m-2} t^k q^{k(m+3)+2k^2}\binom{m-2}{k}_q q^{m-2-k} \\ 
&\quad + \sum_{m \geq 2} \frac{t^m q^{m^2+1}}{(q)_{m-2}} \sum_{k=0}^{m-2} t^{k}q^{k(m+3)+2k^2} \binom{m-2}{k}_q (1-q^{m-2-k}) = C(t,q)
\end{split}
\end{equation*}
proving the third equation in \eqref{eq:functional}.
\begin{equation*}
\begin{split}
\MoveEqLeft
tq B(tq, q) - D(t,q) = \sum_{m \geq 2} \frac{t^m q^{m^2}}{(q)_{m-2}} \sum_{k =1}^{m-2} t^k q^{k(m+1)+2k^2}\binom{m-2}{k}_q (1-q^k) \\ 
&= \sum_{m \geq 3} \frac{t^{m}q^{m^2}}{(q)_{m-2}} \sum_{k = 0}^{m-3} t^{k+1}q^{(k+1)(m+1)+2(k+1)^2}\binom{m-2}{k+1}_q(1-q^{k+1}) \\ 
&= \sum_{m \geq 2} \frac{t^{m+1} q^{m^2+m+3}}{(q)_{m-3}} \sum_{k=0}^{m-3} t^{k} q^{k(m+5)+2k^2} \binom{m-3}{k}_q = tq E(t q^2,q), 
\end{split}
\end{equation*}
proving the fourth equation in \eqref{eq:functional}.
Finally to prove the fifth equation:
\begin{multline*}
t q C(tq,q) = \sum_{m \geq 2} \frac{t^{m+1}q^{m^{2}+m+2}}{(q)_{m-2}} \sum_{k = 0}^{m-2} t^{k} q^{k(m+4)+2k^2} \binom{m-2}{k}_q \\ 
= \sum_{m \geq 3}\frac{t^m q^{m^2-m+2}}{(q)_{m-3}}\sum_{k=0}^{m-3} t^k q^{k(m+3)+2k^2} \binom{m-3}{k}_q = E(t,q).
\end{multline*}
\end{proof}
\begin{rem}
Expanding the q-binomial coefficients with \eqref{eq:binom} and replacing $k_1 = k$, $k_2 = m-k$. We obtain the following \emph{quasi-particle} representations for the above power series:
\begin{equation}
\begin{aligned}
A(t,q) 
&= \sum_{k_1, k_2 \in \Z_{\geq 0}} \frac{t^{2k_1+k_2} q^{4k_1^2 + 3k_1k_2 +
k_2^2 + 2k_1 + 2k_2}}{(q)_{k_1}(q)_{k_2}}, \\
B(t,q) 
 &= t q^2 \sum_{k_1, k_2 \in \Z_{\geq 0}} \frac{t^{2k_1+k_2} q^{4k_1^2 + 3k_1k_2 +
k_2^2 + 5k_1 + 3k_2}}{(q)_{k_1}(q)_{k_2}}, \\
C(t,q) 
 &= t^2 q^5 \sum_{k_1, k_2 \in \Z_{\geq 0}} \frac{t^{2k_1+k_2} q^{4k_1^2 + 3k_1k_2 +
k_2^2 + 9k_1 + 4k_2}}{(q)_{k_1}(q)_{k_2}}, \\
D(t,q)
 &= t^2 q^4 \sum_{k_1, k_2 \in \Z_{\geq 0}} \frac{t^{2k_1+k_2} q^{4k_1^2 + 3k_1k_2 +
k_2^2 + 8k_1 + 4k_2}}{(q)_{k_1}(q)_{k_2}}, \\
E(t,q)
 &= t^3 q^8 \sum_{k_1, k_2 \in \Z_{\geq 0}} \frac{t^{2k_1+k_2} q^{4k_1^2 + 3k_1k_2 + k_2^2 + 11k_1 + 5k_2}}{(q)_{k_1}(q)_{k_2}}.
\end{aligned}
\label{eq:quasi-particle}
\end{equation}
Each of these series, specialized at $t=1$ is a Nahm sum \cite{zagier}
associated to the matrix $\begin{psmallmatrix}8 & 3 \\ 3 & 2\end{psmallmatrix}$
as explained in \ref{no:nahmsum2}.
\end{rem}
\begin{proof}[Proof of Theorem \ref{thm:partition-theorem} \ref{thm:part-c}] By summing \eqref{eq:quasi-particle} we obtain \eqref{eq:pmn.gen.function}.
\end{proof}
\section{Rogers-Ramanujan-Slater Type identities} \label{sec:polynomial-ident}
In this section we prove Theorem \ref{thm:partition-theorem} \ref{thm:part-b} and
prove Theorems \ref{thm:q-ident} and \ref{thm:modules}. The strategy is to first
prove a finite version of these identities. We find families of polynomials
depending on a parameter $n \in \mathbb{Z}_+$. Their limits as $n \rightarrow
\infty$ are the appropriate sides of our $q$-series identities. We prove the
finite version of these identities by finding a recurrence equation satisfied by
both families. This is achieved with the \texttt{Mathematica} package
\texttt{qMultiSum} from RISC
\cite{qmultisum}.

\begin{nolabel}Define the following families of polynomials in $S_n, T_n \in
\mathbb{Z}[q]$. 
\begin{equation}
\begin{aligned}
S_n = S^0_n &= \sum_{k \geq 0} q^{2 k^2} \binom{n-k}{2k}_q, \\
T_n = T^0_n &=
\sum_{m,k\geq 0} q^{m^2+3k m+4k^2} \left( \binom{n-3k-m}{k}_q
\binom{n-4k-m}{m}_q \right. \\ &\quad  - \left.  q^k \binom{n-3k-m-1}{k}_q\binom{n-4k-m-1}{m-1}_q \right).
\end{aligned}
\label{eq:finite-version}
\end{equation}
\label{no:pol-family-def}
\end{nolabel}
\begin{lem}
\begin{align}
\lim_{n \rightarrow \infty} S_n &= \sum_{k \geq 0} \frac{q^{2k^2}}{(q)_{2k}} =
\chi_{\Vir_{3,4}}, \\ 
\lim_{n \rightarrow \infty} T_n &= \sum_{m,k \geq 0} \frac{q^{m^2 + 3k m+
4k^2}}{(q)_k (q)_m}\left( 1 -q^k + q^{k+m} \right). 
\label{eq-limits}
\end{align}
\label{lem:limits}
\end{lem}
\begin{proof}
This is a simple application of the third line of \eqref{eq:alternatives} and 
\[ \lim_{n \rightarrow \infty} \binom{n}{k}_q = \frac{1}{(q)_k}.\]
\end{proof}
\begin{prop}For every $n \geq 0$ we have
\[ S_n = T_n. \]
\label{prop:equality}
\end{prop}
\begin{proof}
In order to prove this proposition we use the \texttt{Mathematica} package
\texttt{qMultiSum}. We first find a recurrence
equation satisfied by the family of polynomials $S_n$. Indeed the following is
true:
\begin{equation}
q^{4n+15}S_n + q^{2n+11}(1+q)(S_{n+3} - S_{n+4}) - q^3 S_{n+5} + 
(1+q+q^2)(qS_{n+6} - S_{n+7}) + S_{n+8} = 0. 
\label{eq:recurrenceS}
\end{equation}
\begin{Verbatim}
In[1]:= <<RISC`qMultiSum`
	qMultiSum Package version 2.54
	written by Axel Riese
	Copyright Research Institute for Symbolic Computation (RISC),
	Johannes Kepler University, Linz, Austria

In[2]:= s[n_,k_] := q^(2k^2) qBinomial[n-k,2k,q]
	strsetsn = qFindStructureSet[s[n,k],{n},{k},{1},{2},2]

Out[3]= {{{2,0},{3,0},{4,0},{4,1},{5,0},{5,1},{6,0},{6,1},{7,0},{7,1},{10,3}},
	{{1,0},{2,0},{3,0},{3,1},{4,0},{4,1},{5,0},{5,1},{6,0},{6,1},{9,3}},
	{{1,0},{2,0},{3,0},{4,0},{4,1},{5,0},{5,1},{6,0},{6,1},{7,0},{7,1},{9,2}},
	{{0,0},{1,0},{2,0},{3,0},{3,1},{4,0},{4,1},{5,0},{5,1},{6,0},{6,1},{8,2}},
	{{0,0},{1,0},{2,0},{3,0},{3,1},{4,0},{4,1},
	{5,0},{5,1},{6,0},{6,1},{7,0},{7,1},{8,1}}}

In[4]:= recsn = qFindRecurrence[s[n, k], {n}, {k}, {1}, {2}, 2, StructSet -> %[[1]]]

Out[4]= {q^(4 n) F[-8+n,-3+k]+q^(9+2 n) (q^3+q^4+q^(2 k)) F[-5+n,-1+k]-
	q^(11+2 k+2 n) F[-5+n,k]-
	q^(9+n) (q^(5+k)+q^(2+n)+q^(3+n)+q^(4+n)+q^(2 k+n)) F[-4+n,-1+k]+
	q^(11+n) (q^(4+k)+q^n+q^(2 k+n)) F[-4+n,k]+
	q^(11+n) (q^(2+k)+q^(3+k)+q^n) F[-3+n,-1+k]-
	q^11 (q^9+q^(2 n)+q^(3+k+n)+q^(4+k+n)) F[-3+n,k]-
	q^(13+k+n) F[-2+n,-1+k]+
	q^14 (q^4+q^5+q^6+q^(k+n)) F[-2+n,k]-q^17 (1+q+q^2) F[-1+n,k]+q^17 F[n,k]==0}

In[5]:= Srec = qSumRecurrence[recsn]

Out[5]= {q^(15+4 n) SUM[n]+q^(11+2 n) (1+q) SUM[3+n]-q^(11+2 n) (1+q) SUM[4+n]-
	q^3 SUM[5+n]+q (1+q+q^2) SUM[6+n]+(-1-q-q^2) SUM[7+n]+SUM[8+n]==0}
\end{Verbatim}
In a similar way we find that the same recurrence \eqref{eq:recurrenceS} is satisfied by $T_n$. It remains to show the initial cases which Mathematica handles
easily.  
\end{proof}
In the same way we prove the following 
\begin{prop}Define the following families of polynomials
\begin{align*}
S^{1/2}_n &= \sum_{k \geq 1} q^{2k^2 -2k} \binom{n-k}{2k-1}_q, \\
S^{1/16}_n &= \sum_{k \geq 0} q^{2k^2+k} \binom{n-k}{2k+1}_q,  \\
T^{1/2}_n &= \sum_{m,k \geq 0} q^{m^2 + 3 k m+ 4k^2+2k} \left(
\binom{n-3k-m-1}{k}_q \binom{n-4k-m-1}{m}_q  \right. \\ 
& \quad - \left. q^{4m+8k+6} \binom{n-3k-m-5}{k}_q \binom{n-4k-m-5}{m}_q \right), \\ 
T^{1/16}_n &= \sum_{m,k \geq 0} q^{m^2 +3 k m +4k^2} \left( q^{k+m} \binom{n-3k-m-1}{k}_q \binom{n-4k-m-1}{m}_q \right. \\ 
&\quad + \left. q^{m+4k+1} \binom{n-3k-m-2}{k}_q \binom{n-4k-m-2}{m}_q \right).
\end{align*}
Then for all $n \geq 1$ we have
\[ S^{1/2}_n = T^{1/2}_n, \qquad S^{1/16}_n = T^{1/16}_n. \]
\label{prop:modules-pol}
\end{prop}
\begin{proof}[Proof of Theorems \ref{thm:q-ident} and \ref{thm:modules}]
Theorem \ref{thm:q-ident} follows directly from Lemma \ref{lem:limits} and
Proposition
\ref{prop:equality}. The other identities in Theorem \ref{thm:modules} follow by
taking the limit $n \rightarrow \infty$ in Proposition \ref{prop:modules-pol}.
\end{proof}
\begin{proof}[Proof of Theorem \ref{thm:partition-theorem} \ref{thm:part-b}] 
The specialization to $t=1$ of \eqref{eq:pmn.gen.function} is given by the RHS
of \eqref{eq:q-ident} which by Theorem \ref{thm:q-ident} equals any of the
alternative forms \eqref{eq:alternatives}. In particular, the last line of
\eqref{eq:alternatives} counts the partitions mentioned in Theorem
\ref{thm:partition-theorem} \ref{thm:part-b}. 
\end{proof}

\section{A PBW basis for the Ising model} \label{sec:PBW}
In this section we finish the proof of Theorem \ref{thm:partition-theorem} and
prove Theorem \ref{thm:main} as a consequence of Theorem \ref{thm:q-ident}. The computations in
this section were carried out with the help of SageMath \cite{sagemath}. For an
implementation of vertex algebras see also \cite{sagevertex}. 

Consider the differential polynomial algebra $\mathbb{C}[L_{-2},L_{-3},\dots]$, with grading defined by $\deg L_{-n} = n$, and the derivation given by $\partial L_{-n} = (n-1) L_{-n-1}$. The \emph{grevlex} monomial ordering on this ring is defined as follows. For a partition $\lambda = [\lambda_1, \lambda_2, \dots, \lambda_m]$ in which $\lambda_1 \geq \lambda_2 \geq \dots \lambda_m \geq 2$ we write \[L_\lambda = L_{-\lambda_1} \cdots L_{-\lambda_m} \in \mathbb{C}[L_{-2},L_{-3},\dots].\] Given two partitions $\lambda, \mu$, we say that $L_\lambda <  L_\mu$ if $\deg L_\lambda < \deg L_\mu$ or if $\deg L_\lambda = \deg L_\mu$ and there exists $i \geq 1$ such that $\lambda_j = \mu_j$ for $0 < j < i$ and $\lambda_i > \mu_i$. Let 
\[ I = \left( a,b \right)_{\partial} \subset \mathbb{C}[L_{-2},L_{-3},\dots], \]
be the differential ideal generated by 
\[ a = L_{-2}^3, \qquad b = L_{2}L_{-3}L_{-4} + \frac{1}{6} L_{-5}L_{-2}^2. \]
The main technical result of this section is the following:
\begin{prop}
For each partition $\lambda$ in \eqref{eq:horrible-set}, there exists an element
of $I$ with leading monomial $L_\lambda$.
\label{prop:technical}
\end{prop}

Theorem \ref{thm:main} and Theorem \ref{thm:partition-theorem} \ref{thm:part-a}
are colloraries of this Proposition:
\begin{proof}[Proof of Theorem \ref{thm:main}] Let $V = \Vir_{3,4}$. We define the differential ideal  \[ I' = \left( \frac{1}{6} L_{-5} L_{-2}^2 +  L_{-4}L_{-3}L_{-2}\right)_{\partial} \subset JR_V,\] and let $A = JR_V/I' $. This is a $\Z_{\geq 0}$ graded differential algebra $A = \oplus_{n \geq 0} A_n$ with graded dimension 
\[ \chi_A (q) = \sum_{n \geq 0} q^n \dim A_n. \]
Since the map \eqref{eq:isomvir} is surjective, we know $\dim A_n \geq \dim V_n$. It remains to show the other inequality. 
Due to Proposition \ref{prop:technical} we have $\dim A_n \leq |\cP(n)|$.
Theorem \ref{thm:q-ident} and Theorem \ref{thm:partition-theorem}
\ref{thm:part-c} imply that $|\cP(n)| = \dim V_n$. 
\end{proof}
\begin{proof}[Proof of Theorem \ref{thm:partition-theorem} \ref{thm:part-a}]
This follows from Theorem \ref{thm:main} since the character of $\Vir_{3,4}$ and
its associated graded $\gr \Vir_{3,4}$ coincide. 
\end{proof}

\begin{proof}[Proof of Proposition \ref{prop:technical}]
We will find the required elements of $I$ by taking appropriate combinations of derivatives of $a$ and $b$ in such a way as to cancel leading terms. The first step is to write formulas for $\partial^{(n)}a$ and $\partial^{(n)}b$ (here $\partial^{(n)} = \tfrac{\partial^n}{n!}$), keeping track of the first few leading terms and the polynomial dependence of their coefficients on $n$. Since our partitions \eqref{eq:horrible-set} have length up to and including $6$, we need to keep track of around $6$ leading terms. 

We prove by induction
\begin{subequations}
\begin{align}
\begin{split}
\partial^{(3k+9)}a &= L_{-5-k}^3 + 6L_{-6-k}L_{-5-k}L_{-4-k} + 3L_{-6-k}^2L_{-3-k} \\ 
&\quad + 3L_{-7-k}L_{-4-k}^2 + 6L_{-7-k}L_{-5-k}L_{-3-k}   
 + 6L_{-7-k}L_{-6-k}L_{-2-k} + \ldots, 
\end{split} \\ 
\begin{split}
\label{eq:subeq2a} \frac{1}{3}\partial^{(3k+10)}a &=L_{-6-k} L_{-5-k}^2 + L_{-6-k}^2L_{-4-k} + 2L_{-7-k}L_{-5-k}L_{-4-k} 
+ 2L_{-7-k}L_{-6-k}L_{-3-k} \\ &\quad + L_{-7-k}^2L_{-2-k} + L_{-8-k}L_{-4-k}^2  
+ 2L_{-8-k}L_{-5-k}L_{-3-k} + 2L_{-8-k}L_{-6-k}L_{-2-k} + \ldots,  
\end{split}\\ 
\begin{split}
\frac{1}{3}\partial^{(3k+11)}a &=L_{-6-k}^2 L_{-5-k} + L_{-7-k}L_{-5-k}^2 + 2L_{-7-k}L_{-6-k}L_{-4-k}  
+ L_{-7-k}^2L_{-3-k}\\
 &\quad  + 2L_{-8-k}L_{-5-k}L_{-4-k} + 2L_{-8-k}L_{-6-k}L_{-3-k}
 + 2L_{-8-k}L_{-7-k}L_{-2-k} + \ldots,
\end{split}
\end{align}
\label{eq:subeq2}
\end{subequations}
where ``$\ldots$'' means terms that are lower in the monomial order. For $0 \leq n \leq 8$ similar expressions for $\partial^{(n)}a$ can be computed by hand. Again we prove by induction
\begin{subequations}
\begin{align}
\begin{split}
\partial^{(3k+6)} b &= \frac{1}{6} \left( 19 k^3+150 k^2+389 k+330 \right) L_{-5-k}^3  \\
&\quad+ \left( 19 k^3+150 k^2+391 k+340 \right) L_{-6-k}L_{-5-k}L_{-4-k}  \\ 
&\quad+   \frac{1}{2} \left(19 k^3+150 k^2+395 k+376  \right) L_{-6-k}^2L_{-3-k}  \\ 
&\quad+ \frac{1}{2} \left( 19 k^3+150 k^2+395 k+344 \right)    L_{-7-k}L_{-4-k}^2  \\ 
&\quad+ \left( 19 k^3+150 k^2+397 k+370 \right)    L_{-7-k}L_{-5-k}L_{-3-k}  \\ 
&\quad+ \left( 19 k^3+150 k^2+403 k+448 \right)   L_{-7-k}L_{-6-k}L_{-2-k} +
\ldots, 
\label{eq:subeq4a}
\end{split}\\
\begin{split}
\partial^{(3k+7)}b &=
\frac{1}{2}(19k^3 + 169k^2 + 496k + 480) {L_{-6-k} L_{-5-k}^2}  \\
&\quad +\frac{1}{2}(19k^3 + 169k^2 + 498k + 496) {L_{-6-k}^2  L_{-4-k}}  \\
&\quad +\left(19k^3 + 169k^2 + 500k + 496 \right) {L_{-7-k} L_{-5-k} L_{-4-k}}  \\
&\quad +\left(19k^3 + 169k^2 + 504k + 544\right) {L_{-7-k} L_{-6-k} L_{-3-k}}  \\
&\quad +\frac{1}{2}(19k^3 + 169k^2 + 512k + 640) {L_{-7-k}^2 L_{-2-k}}  \\
&\quad +\frac{1}{2}(19k^3 + 169k^2 + 506k + 496) {L_{-8-k}  L_{-4-k}^2}  \\
&\quad +\left(19k^3 + 169k^2 + 508k + 528\right) {L_{-8-k} L_{-5-k} L_{-3-k}}  \\
&\quad +\left(19k^3 + 169k^2 + 514k + 624\right) {L_{-8-k} L_{-6-k} L_{-2-k}} +
\ldots , 
\end{split} \\
\begin{split}
\partial^{(3k+8)}b &=\frac{1}{2} \left(19 k^3+188 k^2+615 k+666\right)  L_{-6-k}^2L_{-5-k} \\
&\quad+   \frac{1}{2} \left(19 k^3+188 k^2+617 k+672\right)  L_{-7-k}L_{-5-k}^2  \\ 
&\quad+  \left( 19 k^3+188 k^2+619 k+694\right) L_{-7-k}L_{-6-k}L_{-4-k}  \\ 
&\quad+  \frac{1}{2} \left(19 k^3+188 k^2+625 k+760\right)   L_{-7-k}^2L_{-3-k}  \\ 
&\quad+  \left( 19 k^3+188 k^2+623 k+690 \right)   L_{-8-k}L_{-5-k}L_{-4-k}  \\ 
&\quad+ \left( 19 k^3+188 k^2+627 k+750 \right)    L_{-8-k}L_{-6-k}L_{-3-k}  \\ 
&\quad+  \left(19 k^3+188 k^2+635 k+870 \right)   L_{-8-k}L_{-7-k}
L_{-2-k}+ \ldots .
\end{split} 
\end{align}
\label{eq:subeq4}
\end{subequations}
Examining the expressions for $\partial^{(n)}a$ in equations 
\eqref{eq:subeq2} we see that for partitions $\lambda$ which contain $[p, p, p]$, etc. (the partitions in first line of \eqref{eq:horrible-set}), the monomials $L_\lambda$ are leading monomials of elements of $I$. From \eqref{eq:subeq2a} and \eqref{eq:subeq4a} we see that the polynomial 
\[ r_k =  \partial^{(3k+1)}b - \frac{1}{6} (19k^3+55k^2+48k+12) \partial^{(3k+4)}a, \]
has $L_{-4-k}^2L_{-2-k}$ as leading monomial.
Similarly the polynomials
\[ s_k =  \partial^{(3k+2)}b - \frac{1}{6}(19k^3+74k^2+91k+36)\partial^{(3k+5)}a,   \]
\[ t_k =  \partial^{(3k)}b - \frac{1}{6}(19k^3+36k^2+17k)\partial^{(3k+3)}a, \]
have $L_{-5-k}L_{-3-k}^2$ and  $L_{-4-k}L_{-3-k}L_{-2-k}$ respectively as leading monomials. These show that the monomials corresponding to all elements of length three in  \eqref{eq:horrible-set} are leading monomials of elements of $I$. We now define
\begin{align*}
u_0 &= 8 L_{-5} t_0 - 6 L_{-2}t_1, \\ 
\begin{split} u_{k+1} &=  
(2k+10)L_{-6-k} t_{k+1} - (2k+8) L_{-3-k} t_{k+2}  \\
&\quad -\frac{(2k+10)(3k+20)}{3}L_{-2-k} \partial^{(3k+10)}a,
\end{split} && k \geq 0,\\
\begin{split}
v_k &= \left( 11 k^3+318 k^2+3061 k+9426\right)L_{-2-k}t_{k+2}  
- \left(7 k^3+90 k^2+349 k+370\right) L_{-6-k} s_{k} \\ 
&\quad -\left(11 k^3+191 k^2+1029 k+1745\right) L_{-3-k}s_{k+1} 
-8 \left(k^3+19 k^2+121 k+255\right) L_{-4-k}r_{k+1} \\ 
&\quad +\frac{1}{3}\left(35 k^4+709 k^3+5075 k^2+14763 k+13690 \right) L_{-6-k}\partial^{(3k+5)}a \\
&\quad +\frac{8}{3} \left(k^4+26 k^3+254 k^2+1102 k+1785\right) L_{-3-k} \partial^{(3k+8)}a, 
\end{split} && k \geq 0 \\ 
w_k &= (k+2)L_{-6-k}r_k - (k+6)L_{-2-k}s_{k+1}, && k \geq 0 \\
 y_0 &= 42 L_{-5}r_0 - 84 L_{-2} \partial^{(7)}a - 12 L_{-2} r_1  + 108 L_{-6} t_0,  \\ 
\begin{split}
y_1&= -640 L_{-6}r_1 + 2584 L_{-3} r_2 - 4480 L_{-5} s_1 + \frac{81856}{3}
L_{-2} s_2\\ 
&\quad + 1216 L_{-7} t_1 
- 10304 L_{-4} t_2 + 72128 L_{-7}
\partial^{(6)}a  +\frac{112000}{3} L_{-6} \partial^{(7)}a,
\end{split} \\ 
\begin{split}
y_{k+2} &= 2(-k^4-23 k^3-189 k^2-657 k-810) L_{-7-k}r_{k+2}   \\
&\quad + 
(-7 k^4 - 162 k^3 - 1129 k^2 - 2198 k + 1360) L_{-4-k}r_{k+3} \\ 
& \quad -4 (k^4+28 k^3+289 k^2+1302 k+2160) L_{-6-k}s_{k+2} \\ 
& \quad +\frac{16 (13 k^4+384 k^3+3849 k^2+14962 k+16600)}{(k+4)} L_{-3-k}s_{k+3} \\ 
&\quad -(5 k^4+162 k^3+1911 k^2+7850 k+4880) L_{-8-k}t_{k+2}  \\ 
&\quad -(7 k^4 + 207 k^3 + 2265 k^2 + 10841 k + 19080)L_{-5-k}t_{k+3} \\ 
&\quad +\Bigl(21 k^5+691 k^4+8865 k^3 
+55173 k^2+165650 k+190800\Bigr) L_{-8-k} \partial^{(3k+9)}a  \\ 
&\quad -2 (k^5+47 k^4+901 k^3+8393 k^2+37218 k+62640) L_{-2-k}\partial^{(3k+15)}a  \\ 
&\quad +\frac{2}{3} \left(9 k^5+311 k^4+4253 k^3+28769 k^2 
+96258
k+127440\right)L_{-7-k}\partial^{(3k+10)}a 
\end{split} && k \geq 0.
\end{align*}
We see that the leading monomial of $u_k$, $v_k$, $w_k$ and $y_k$ are respectively
\[
\begin{aligned}
&L_{-5-k}^2L_{-2-k}^2, \qquad &L_{-6-k}^2L_{-3-k}L_{-2-k}, \\
&L_{-6-k}L_{-5-k}L_{-3-k}L_{-2-k}, \qquad  &L_{-8-k}L_{-7-k}L_{-4-k}^2.
\end{aligned}
\]
This shows that the monomials corresponding to all elements in the second group in \eqref{eq:horrible-set}, are leading monomials of elements in $I$. 
The leading monomial of 
\[ z_k = (k+6)L_{-2-k} y_{k+2}  - 32 \left(4 k^4+165 k^3+2427 k^2+14184 k+27440\right) L_{-8-k}L_{-7-k}r_k,\]
is $L_{-8-k}L_{-7-k}L_{-5-k}L_{-3-k}L_{-2-k}$, the infinite family of length five in \eqref{eq:horrible-set}. 
Defining 
\begin{align*}
e_1 &= 3L_{-2}s_0 - L_{-5} \partial^{(2)}a, \\
e_2 &=
L_{-6}y_0 + 384 L_{-2}^2 \partial^{(11)}a - 832 L_{-2}^2 s_2 - 12 L_{-7} u_0,  \\
e_3 &= 432 L_{-2}w_1 + 11520 L_{-7} L_{-6} b + 73 L_{-7} y_0 + 53088 L_{-7}^2
\partial^{(2)}a, \\ 
e_4 &= L_{-9}L_{-8}u_0 + 8 L_{-9}L_{-8}L_{-6}\partial^{(2)}a +
\frac{112}{1415040} L_{-2}^2 y_3, 
\end{align*}
we see that the leading monomials for these elements are respectively 
\[ 
\begin{aligned}
&L_{-5}L_{-4}L_{-2}L_{-2}, \qquad &L_{-7}L_{-6}L_{-4}L_{-2}L_{-2}, \\ 
&L_{-7}L_{-7}L_{-4}L_{-2}L_{-2}, \qquad   &L_{-9}L_{-8}L_{-6}L_{-4}L_{-2}L_{-2}.
\end{aligned}
\]
The four remaining cases in \eqref{eq:horrible-set}. 
\end{proof}
\begin{rem} In fact it can be shown that three of the four exceptional cases in \eqref{eq:horrible-set} already lie in the differential ideal $I$. For example 
\[
\frac{1}{204}\left( 3 L_{-2} \partial^2 b - 18 L_{-4}b - 19 L_{-5} \partial^2 a
- 88 L_{-6} \partial a - 60 L_{-7}a  \right)   =  L_{-5}L_{-4}L_{-2}L_{-2}.
\]
A similar computation shows that $L_{-7}L_{-6}L_{-4}L_{-2}^2$ and $L_{-9}L_{-8}L_{-6}L_{-4}L_{-2}^2$ lie in $I$. The smallest expressions we found for these generators are too long to fit in these pages.
\label{rem:in-ideal}
\end{rem}
\begin{cor} The set consisting of the infinite family of polynomials $\partial^{(k)}a$,
$r_k$, $s_k$, $t_k$, $u_k$, $v_k$, $y_k$, $z_k$, $k \geq 0$ together with the
four exceptional elements $e_i$, $i=1,\dots,4$,
forms a Gr\"{o}bner basis of the ideal 
\[ I = \left( L_{-2}^3, \frac{1}{6} L_{-5}L_{-2}^2 +  L_{-4}L_{-3}L_{-2}
\right)_{\partial} \subset \mathbb{C}[L_{-2},L_{-3},\dots]. \]
\label{cor:grobner}
\end{cor}
\section{Conclusion}
In this article we proved three new $q$-series identities in Theorem \ref{thm:modules}. These imply that linear combinations of Nahm sums for a given matrix may be modular invariant, even when no single summand is. As a consequence of one of these identities we described the singular support of the Ising model vertex algebra. We showed that it is, in a differential sense, a \emph{hypersurface} of the arc space of its associated scheme, that is, it is defined as the zero set of a single equation and all of its derivatives. This is the first known example of a vertex algebra whose singular support is a proper subscheme of the arc space of its associated scheme defined by finitely many equations and their derivatives. We have also found an explicit basis for the Ising model consisting on monomials in the Virasoro Lie algebra generators applied to the vacuum vector. 

The methods of this article can be applied to other Virasoro minimal models and their modules to produce new $q$-series identities and their combinatorial interpretations. 
\def\cprime{$'$}

\end{document}